\documentclass{amsart}

\usepackage{amsmath,amssymb,amsthm}

\usepackage{graphicx,tikz, caption}

\newtheorem{theorem}{Theorem}
\newtheorem*{thm}{Theorem}
\newtheorem{corollary}{Corollary}
\newtheorem*{proposition}{Proposition}

\newtheorem*{lemma}{Lemma}

\theoremstyle{definition}

\theoremstyle{remark}

\begin{document}

\title[]{Discrete Rearrangements and the\\ P\'olya-Szeg\H{o} Inequality on Graphs}

\author[]{Stefan Steinerberger}
\address{Department of Mathematics, University of Washington, Seattle, WA 98195, USA} \email{steinerb@uw.edu}

\thanks{S.S. is supported by the NSF (DMS-2123224) and the Alfred P. Sloan Foundation.}

\begin{abstract} For any $f: \mathbb{R}^n \rightarrow \mathbb{R}_{\geq 0}$ the symmetric decreasing rearrangement $f^*$ satisfies the P\'olya-Szeg\H{o} inequality $\| \nabla f^*\|_{L^p} \leq \| \nabla f\|_{L^p}$. The goal of this paper is to establish analogous results in the discrete setting for graphs satisfying suitable conditions. We prove that if the edge-isoperimetric problem on a graph has a sequence of nested minimizers, then this sequence gives rise to a rearrangement satisfying the P\'olya-Szeg\H{o} inequality in $L^1$. This shows, for example, that a specific rearrangement on the grid graph $\mathbb{Z}^2$, going around the origin in a spiral-like manner, satisfies $\| \nabla f^*\|_{L^1} \leq \| \nabla f\|_{L^1}$.
The $L^{\infty}-$case is implied by an optimal ordering condition in vertex-isoperimetry.  We use these ideas to prove
that the canonical rearrangement on the infinite $d-$regular tree satisfies the P\'olya-Szeg\H{o} inequality for all $1 \leq p \leq \infty$.
\end{abstract}

\maketitle

\vspace{0pt}
    
    \section{Introduction}
  \subsection{Rearragements.} Rearrangement principles are a cornerstone of analysis. If $f: \mathbb{R}^n \rightarrow \mathbb{R}_{\geq 0}$
  is a nonnegative function, then its symmetric decreasing rearrangement $f^*:\mathbb{R}^n \rightarrow \mathbb{R}_{\geq 0}$ is defined by asking that
  \begin{enumerate}
  \item the super-level sets $\left\{ x \in \mathbb{R}^n: f^*(x) \geq  s\right\}$ are balls centered at the origin 
  \item which have the same measure as the original super-level set
  $$ \left|\left\{ x \in \mathbb{R}^n: f(x) \geq  s\right\}\right| = \left|\left\{ x \in \mathbb{R}^n: f^*(x) \geq  s\right\}\right|.$$
  \end{enumerate}
If one looks at the domains where $f$ and $f^*$ assume values in a certain interval $[a,b]$, then these domains
  have the exact same volume. This implies that $ \| f\|_{L^p} = \|f^*\|_{L^p}$ for all $p>0$ (sometimes known as `layer cake formula' or `bathtub principle').
The celebrated P\'olya-Szeg\H{o} inequality implies that rearrangement is `smoothing' in the sense of decreasing the size of the gradient and, for all $1 \leq p \leq \infty$,
  $$ \| \nabla f^* \|_{L^p}  \leq  \| \nabla f \|_{L^p}.$$
 This property has many applications in analysis, partial differential equations and mathematical physics, we refer to the excellent books by Baernstein \cite{baernstein}, Lieb-Loss \cite{lieb} and P\'olya-Szeg\H{o} \cite{Polya}. 
 One particularly important application concerns certain partial differential equations. Consider, for example, the equation $-\Delta u = u^p$ on $\mathbb{R}^n$. It arises naturally as the Euler-Lagrange equation of the energy functional
 $$ J(u) = \int_{\mathbb{R}^{n}}  \frac{1}{2} |\nabla u(x)|^2 - \frac{u(x)^{p+1}}{p+1}~ dx.$$
 If the functional $J$ assumes a global minimum, then global minimum will solve $-\Delta u = u^p$.  Applying the P\'olya-Szeg\H{o} inequality implies that $J(u^*) \leq J(u)$ showing that the existence of a minimum implies the existence of a radially symmetric solution of the partial differential equation. 
    
 \subsection{Graphs} One could ask whether similar things are possible on combinatorial graphs $G=(V,E)$. Here, $V$ is the set of vertices which we always assume to be countable, $E \subset V \times V$ is the set of edges. For our results to be meaningful, we always require that the graphs have locally finite degree.
 We recall that if
 $f:V \rightarrow \mathbb{R}$, then the $L^p-$norm of the function $f$ and its derivative $\nabla f$ are defined as
$$ \|f\|_{L^p}^p = \sum_{v \in V} |f(v)|^p \qquad \mbox{and} \qquad \| \nabla f\|^p_{L^p} = \sum_{(v,w) \in E} |f(v) - f(w)|^p.$$
 Our definition of a rearrangement on a graph follows the approach of Pruss \cite{pruss}. Given a graph $G=(V,E)$, possibly infinite, a rearrangement is a permutation of the vertices $v_1, v_2, \dots $. Having fixed such a permutation of the vertices, the rearrangement procedure for any given non-negative function $f: V \rightarrow \mathbb{R}_{\geq 0}$ defines a new function $f^*:V \rightarrow \mathbb{R}_{\geq 0}$ via
  $$f^*(v_k) = \mbox{the}~k-\mbox{th largest value assumed by}~f.$$  
  This means that $f^*$ always assumes its largest value in $v_1$, its second largest value in $v_2$ and so on.
 This construction automatically implies
   $ \|f^*\|_{L^p} = \| f\|_{L^p}$
   for all possible rearragements: the remaining question is whether there are rearrangements leading to `smooth' functions as in  $\| \nabla f^*\|_{L^p}$  being smaller or at the very least not much larger than  $\| \nabla f\|_{L^p}$. We also note the pointwise inequality $\nabla |f| \leq |\nabla f|$.

  \begin{center}
  \begin{figure}[h!]
  \begin{tikzpicture}[scale=0.7]
\node at (0,0) {\includegraphics[width=0.4\textwidth]{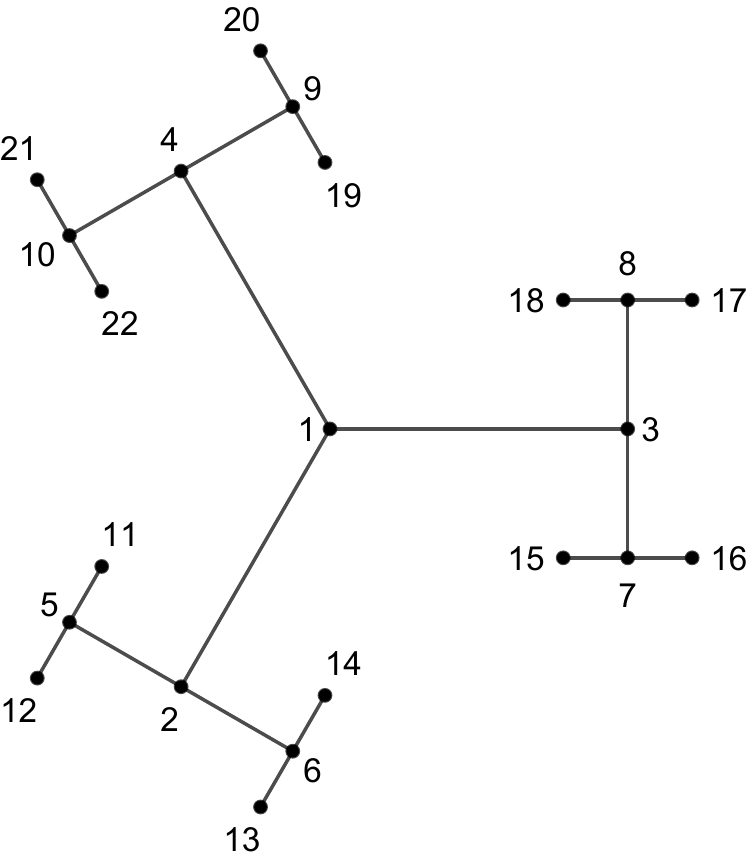}};
  \end{tikzpicture}
    \caption{Pruss showed that for the canonical rearrangement on the $d-$regular infinite tree one has $\| \nabla f^*\|_{L^2} \leq \| \nabla f\|_{L^2}$. Here: $d=3$ and only the first three layers shown.}
    \label{fig:tree}
  \end{figure}
  \end{center}
  
  The study of this problem can perhaps be said to have been initiated by Hardy-Littlewood \cite{hardy} in their work on (discrete) rearrangement inequalities on the integer lattice $\mathbb{Z}$ and the half-line $\mathbb{N}$. The conceptual leap to general graphs seems to be due to Pruss \cite{pruss} who showed that the canonical rearrangement on the infinite $d-$regular tree (see Fig. 1) satisfies the P\'olya-Szeg\H{o} inequality for $p=2$. Pruss obtains more general results, in particular a Riesz convolution-rearrangement inequality which then implies the P\'olya-Szeg\H{o} inequality for $p=2$. An example on the lattice $\mathbb{Z}$ is shown in Fig. \ref{fig:two}.    
      \begin{center}
  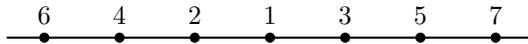
\begin{figure}[h!]
  \begin{tikzpicture}[scale=1]
\draw [thick] (0,0) -- (7,0);
\filldraw (3.5, 0) circle (0.06cm);
\filldraw (4.5, 0) circle (0.06cm);
\filldraw (2.5, 0) circle (0.06cm);
\filldraw (5.5, 0) circle (0.06cm);
\filldraw (1.5, 0) circle (0.06cm);
\filldraw (6.5, 0) circle (0.06cm);
\filldraw (0.5, 0) circle (0.06cm);
\node at (0.5, 0.3) {6};
\node at (1.5, 0.3) {4};
\node at (2.5, 0.3) {2};
\node at (3.5, 0.3) {1};
\node at (4.5, 0.3) {3};
\node at (5.5, 0.3) {5};
\node at (6.5, 0.3) {7};
  \end{tikzpicture}
    \caption{A rearrangement on the lattice $\mathbb{Z}$ satisfying the P\'olya-Szeg\H{o} inequality $\| \nabla f^*\|_{L^p} \leq \| \nabla f\|_{L^p}$ for all $1 \leq p \leq \infty$ (see \cite{haj}).}
    \label{fig:two}
  \end{figure}
  \end{center}
  
 \vspace{-20pt} 
  
Few such results are available \cite{biy, gar, gupta, haj, hhh, pruss, pruss2}; this is perhaps not too surprising, the P\'olya-Szeg\H{o} inequality reflects the overall symmetry of the Euclidean ball. One would perhaps not expect to be able to find counterparts of such results in the general discrete setting outside of a few special cases: the symmetries of the continuous setting are difficult to replace. In much the same vein, various rearrangement inequalities in the continuous setting are known to hold on $\mathbb{R}^n, \mathbb{S}^n, \mathbb{H}^n$ but not much is known (or expected to be true) on general manifolds.

      \begin{center}
  \begin{figure}[h!]
  \begin{tikzpicture}[scale=0.7]
\filldraw (0.5, 0) circle (0.06cm);
\filldraw (1.5, 0) circle (0.06cm);
\filldraw (2.5, 0) circle (0.06cm);
\filldraw (0.5, 1) circle (0.06cm);
\filldraw (1.5, 1) circle (0.06cm);
\filldraw (2.5, 1) circle (0.06cm);
\filldraw (0.5, 2) circle (0.06cm);
\filldraw (1.5, 2) circle (0.06cm);
\filldraw (2.5, 2) circle (0.06cm);
\draw  (0,0) -- (3,0);
\draw  (0,1) -- (3,1);
\draw  (0,2) -- (3,2);
\draw  (0.5,-0.5) -- (0.5,2.5);
\draw  (1.5,-0.5) -- (1.5,2.5);
\draw  (2.5,-0.5) -- (2.5,2.5);
  \end{tikzpicture}
    \caption{A part of the standard grid graph $(\mathbb{Z}^2, \ell^1)$.}
    \label{fig:gridoo}
  \end{figure}
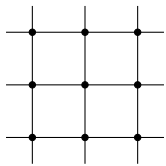
  \end{center}
  \vspace{-10pt}
The impossibility of a P\'olya-Szeg\H{o} inequality was recently very clearly demonstrated for the grid graph $(\mathbb{Z}^2, \ell^1)$. The vertices of this graph are $\mathbb{Z}^2$ and two vertices are connected if their coordinates differ in one entry by 1 (see Fig. \ref{fig:gridoo}). 

\begin{thm}[Hajaiej, Han \& Hua \cite{hhh}] There does not exist a rearrangement of $(\mathbb{Z}^2, \ell^1)$ such that for all functions
$$ \| \nabla f^* \|_{L^2} \leq \| \nabla f \|_{L^2}.$$
\end{thm}

We give a short proof of the general non-existence for $1 < p < \infty$ (inspired by the idea in \cite{hhh}) in \S \ref{sec:short}. It would be interesting if this failure of P\'olya-Szeg\H{o} could be made quantitative: what is the largest constant $\delta > 0$ such that for every rearrangement there exists a function with $ \| \nabla f^* \|_{L^2} > (1+\delta) \| \nabla f \|_{L^2}$? One could also wonder about substitute results: for example, we prove that there exists a rearrangement satisfying $\| \nabla f^*\|_{L^2} \leq \sqrt{2} \cdot \|\nabla f\|_{L^1}$.
It suggests a general question.

\begin{quote}
\textbf{Question.} When does a graph admit a rearrangement satisfying $ \| \nabla f^* \|_{L^p} \leq c_p \| \nabla f \|_{L^p}$? How small can $c_p$ be?  When is $c_p = 1$?
\end{quote}

The purpose of our paper is to show that the endpoint cases are quite interesting and naturally related to the edge-isoperimetric problem ($p=1$)
and the vertex-isoperimetric problem ($p=\infty$). (Note added in print: we have since considered this question on the lattice graph $\mathbb{Z}^d$ in joint work with Shubham Gupta \cite{shubham}).

    \section{Main Results}
    \subsection{Spiral Rearrangement} 
We were motivated by the spiral rearrangement illustrated in Fig. \ref{fig:spiral} and Fig. \ref{fig:example}.  It rearranges a function $f:\mathbb{Z}^2 \rightarrow \mathbb{R}_{\geq 0}$ by sending its largest value to the vertex with the label 1, the second-largest is sent to the vertex with label 2 and so on in a decreasing fashion spiraling around the origin. The choice of spiraling counterclockwise as opposed to clockwise was arbitrary. It seems like a good rearrangement on the grid graph $\mathbb{Z}^2$.

  \begin{center}
  \begin{figure}[h!]
  \begin{tikzpicture}[scale=1.4]
  \node at (0,0) {1};
    \node at (0.5,0) {2};
    \node at (0.5,0.5) {3};
  \node at (0,0.5) {4};
  \node at (-0.5,0.5) {5};
  \node at (-0.5,0) {6};
  \node at (-0.5,-0.5) {7};
  \node at (0,-0.5) {8};
  \node at (0.5,-0.5) {9};
    \node at (1,-0.5) {10};
    \node at (1,0) {11};
    \node at (1,0.5) {12};
    \node at (1,1) {13};
        \node at (0.5,1) {14};
    \node at (0,1) {15};
    \node at (-0.5,1) {16};
  \draw [->] (1 ,-0.35) -- (1,-0.15);
    \draw [->] (1 ,0.15) -- (1,0.35);
    \draw [->] (1 ,1.3/2) -- (1,1.7/2);
    \draw [->] (1.7/2,2/2) -- (1.4/2,2/2);
        \draw [->] (0.7/2,2/2) -- (0.4/2,2/2);
    \draw [->] (-0.3/2,2/2) -- (-0.6/2,2/2);
        \draw [->] (-1.3/2,2/2) -- (-1.8/2,2/2);
  \draw [->] (0.15, 0) -- (0.35, 0);
    \draw [->] (1/2, 0.3/2) -- (1/2, 0.7/2);
  \draw [->] ( 0.7/2,1/2) -- (0.3/2,1/2);
 \draw [->] ( -0.3/2,1/2) -- (-0.7/2,1/2);
 \draw [->] ( -1/2,0.7/2) -- (-1/2,0.3/2);
 \draw [->] ( -1/2,-0.3/2) -- (-1/2,-0.7/2);
 \draw [->] ( -0.7/2,-1/2) -- (-0.3/2,-1/2);
  \draw [->] ( 0.3/2,-1/2) -- (0.7/2,-1/2);
  \draw [->] ( 1.3/2 ,-1/2) -- (1.7/2,-1/2);
    \end{tikzpicture}
    \vspace{-10pt}
    \caption{The spiral rearrangement on $\mathbb{Z}^2$.}
    \label{fig:spiral}
  \end{figure}
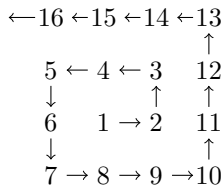
  \end{center}

\vspace{-10pt}

Our first result makes this intuition precise, the spiral rearrangement is a good way of rearranging non-negative functions (meaning functions $f:\mathbb{Z}^2 \rightarrow \mathbb{R}_{\geq 0} = [0,\infty]$): it satisfies the P\'olya-Szeg\H{o} inequality $\| \nabla f^*\|_{L^1} \leq \| \nabla f\|_{L^1}$.
  \begin{theorem}
  Let $f:\mathbb{Z}^2 \rightarrow \mathbb{R}_{\geq 0}$ and let $f^*$ denotes its spiral rearrangement. Then
  $$ \| \nabla f^*\|_{L^1} \leq  \| \nabla f\|_{L^1} \qquad \mbox{and} \qquad  \| \nabla f^*\|_{L^{\infty}} \leq  2 \cdot \| \nabla f\|_{L^{\infty}}.$$
  \end{theorem}
Using vector interpolation $\| v\|_{\ell^p} \leq \| v\|_{\ell^1}^{1/p} \cdot \|v\|_{\ell^{\infty}}^{1-1/p}$, we can combine these inequalities to deduce that the rearrangement is not too badly behaved in other norms
$$\| \nabla f^*\|_{L^{p}} \leq 2^{1-1/p} \cdot  \| \nabla f\|_{L^{1}} \quad \mbox{and} \quad 1 \leq p \leq \infty.$$
This leads to a number of interesting questions: is it possible to prove an estimate $ \| \nabla f^*\|_{L^p} \leq c_p  \cdot \| \nabla f\|_{L^p}$? Is $c_p \leq 2$? Which rearrangement has the smallest constant $c_p$? We refer to \cite{shubham} for some progress on these questions.
    \begin{center}
  \begin{figure}[h!]
  \begin{tikzpicture}[scale=1]
  \node at (2,0.5) {\includegraphics[width=0.25\textwidth]{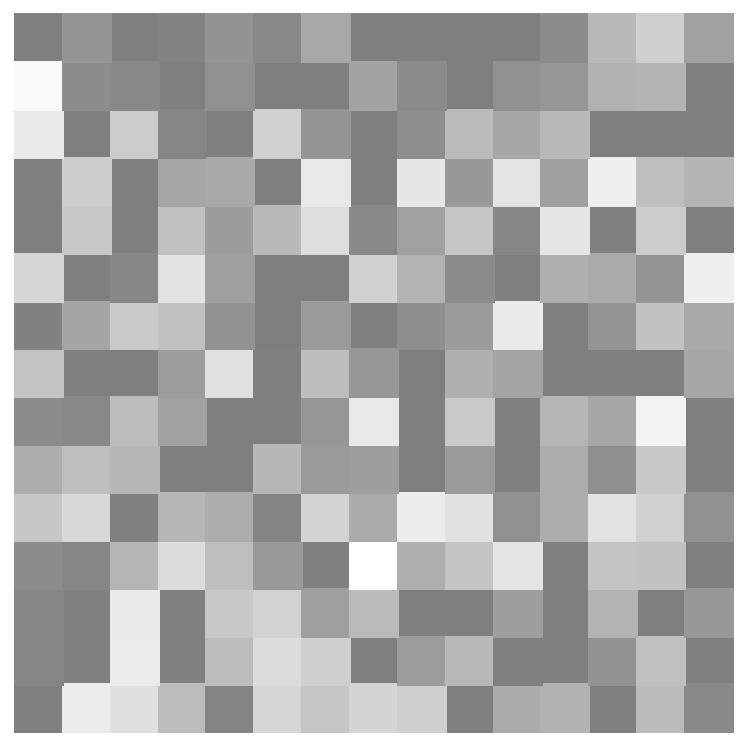}};
    \node at (7,0.5) {\includegraphics[width=0.25\textwidth]{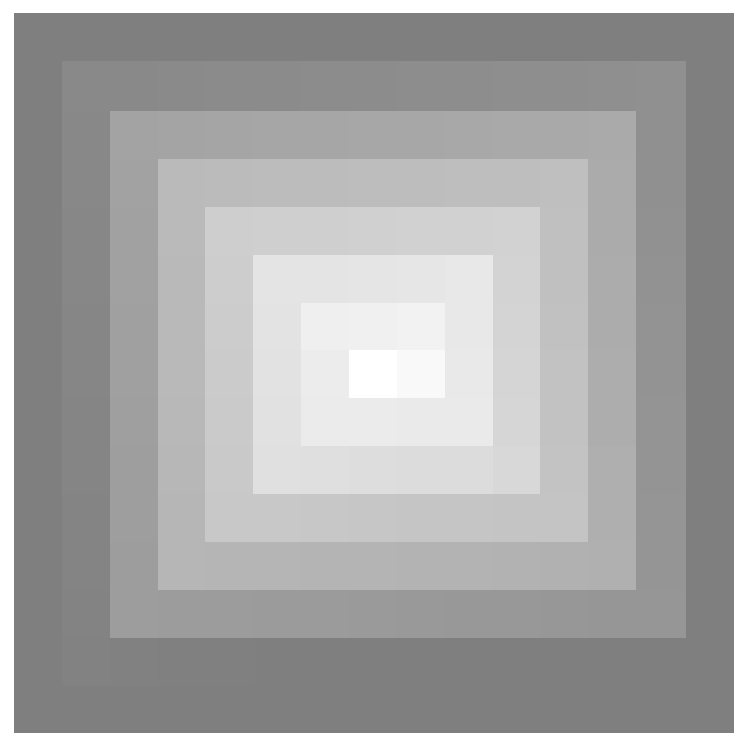}};
    \end{tikzpicture}
    \caption{Left: an example of a function $f:\mathbb{Z}^2 \rightarrow \mathbb{R}_{\geq 0}$ compactly supported around the origin (larger values are brighter). Right: the same function rearranged using the spiral rearrangement.}
      \label{fig:example}
  \end{figure}
  \end{center}
  \vspace{-10pt}
  
  Theorem 1 also suggests an interesting way to think about rearrangements on graphs in general: instead of asking for the full P\'olya-Szeg\H{o} inequality (and thus, implicitly, require a great a deal of underlying symmetry) one could instead try to look for `reasonable' rearrangements satisfying  $ \| \nabla f^*\|_{L^p} \leq c_p \| \nabla f\|_{L^p}$ with the constant $c_p \geq 1$ being as small as possible.   
  \subsection{$L^1-$P\'olya-Szeg\H{o} inequality.}
  We now provide a more general framework for the $L^1-$inequality.  As is well understood in the continuous setting, the quantity $\| \nabla f\|_{L^1}$
  is related to the level sets of the function $f$ via the coarea formula \cite{coarea} and this establishes a natural connection to isoperimetry. The same kind of argument works in the discrete setting: the relevant geometric notion will be edge-isoperimetry. For any subset $A \subseteq V$, the edge boundary $\partial_E(A)$ is defined as
$$ \partial_E (A)  = \left\{e \in E: e~\mbox{runs between}~A~\mbox{and}~V \setminus A \right\}.$$
The problem of minimizing the size of the edge boundary $\# \partial_E   (\left\{v_1, \dots, v_k\right\}) $  among all sets of $k$ vertices, also known as the edge-isoperimetric problem, is well-studied in a variety of settings (see, for example, Bollob\'as \& Leader \cite{boll, boll2}, 
Harper \cite{harper}, Lindsey \cite{lindsey}, Tillich \cite{tillich} and the survey of Bezrukov \cite{bez}).

      \begin{center}
  \begin{figure}[h!]
  \begin{tikzpicture}[scale=0.8]
  \foreach \x in {1,2,...,5}
    \foreach \y in {2,3,...,5}
      { \filldraw (\x,\y) circle (0.06cm);
        }
%    \draw (1, 1.5) -- (1, 5.5);
%    \foreach \x in {2,...,5}
%    { 
%    \draw (0.5, \x) -- (5.5, \x);
%      \draw (\x, 1.5) -- (\x, 5.5);
%    }      
    \draw (1, 1.5) -- (1, 5.5);
        \draw (2, 1.5) -- (2, 3);
            \draw[dashed] (2, 3) -- (2, 5);
 \draw (2, 5) -- (2, 5.5);
  \draw (3, 1.5) -- (3, 2);
            \draw[dashed] (3, 2) -- (3, 5);
            \draw (3,3) -- (3,4);
             \draw (4,3) -- (4,4);
 \draw (3, 5) -- (3, 5.5);
   \draw (4, 1.5) -- (4, 2);
            \draw[dashed] (4, 2) -- (4, 5);
 \draw (4, 5) -- (4, 5.5);
 \draw (0.5, 2) -- (5.5, 2);
  \draw (0.5, 5) -- (5.5, 5);
  \draw (0.5, 3) -- (2, 3);
  \draw [dashed] (2,3) -- (5, 3);
    \draw  (3,3) -- (4, 3);
      \draw (0.5, 4) -- (1, 4);
  \draw [dashed] (1,4) -- (5, 4);
    \draw  (2,4) -- (4, 4);
    \draw (5, 3) -- (5.5, 3);
    \draw (5, 1.5) -- (5, 5.5);
    \filldraw (3,3) circle (0.15cm);
        \filldraw (4,3) circle (0.15cm);
    \filldraw (4,4) circle (0.15cm);
    \filldraw (3,4) circle (0.15cm);
        \filldraw (2,4) circle (0.15cm);
  \draw [thick] (3, 2.3) to[out=0, in=270] (4.6, 3.5) to[out=90, in=0] (3, 4.6) to[out=180, in=90] (1.5, 4) to[out=270, in=180] (2.3, 3.3) to[out=0, in=180] (3, 2.3);      
  \end{tikzpicture}
    \caption{5 vertices (highlighted) with $\# \partial_E(A) = 10$. $A$ solves the edge-isoperimetric problem: any set of 5 vertices in $(\mathbb{Z}^2, \ell^1)$ has a at least 10 edges (dashed) connecting it to the complement.}
    \label{fig:edge}
  \end{figure}
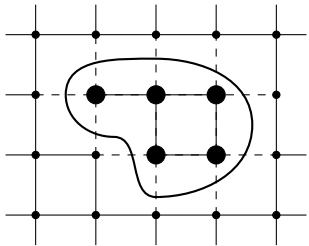
  \end{center}
  
  \vspace{-0pt}

Theorem 2 will prove that if solutions to the edge-isoperimetric problem are nested, meaning that an edge-isoperimetric set for $k+1$ elements can be obtained by adding a suitable vertex to an edge-isoperimetric set with $k$ elements, then such a nested set of extremizers correspond to a rearrangement satisfying the $L^1-$P\'olya-Szeg\H{o} inequality  $\| \nabla f^*\|_{L^1} \leq  \| \nabla f\|_{L^1}$. The proof shows slightly more: `nearly' optimal sets `nearly' imply the inequality in $L^1$ (however, this extension will not be used anywhere else in the paper).

\begin{theorem}[$L^1-$P\'olya-Szeg\H{o}] Let $v_1, v_2, \dots$ be a permutation of the vertices $V$ such that, for some $\alpha \geq 1, \beta \geq 0$ and all $N \in \mathbb{N}$,  
$$ \# \partial_E (\left\{v_1, \dots, v_N\right\}) \leq \beta +   \alpha  \inf_{ A \subseteq  V \atop  \# A =N}~~ \#~ \partial_E (A).$$
Then the associated rearrangement satisfies
  $$ \| \nabla f^*\|_{L^1} \leq  \alpha \| \nabla f\|_{L^1} +  \beta \|f\|_{L^{\infty}}.$$
\end{theorem}
The best possible case is $\alpha =1$ and $\beta = 0$.
Several of the known edge-isoperimetric sets in different graphs such as $(\mathbb{Z}^d, \ell^1)$ or $(\mathbb{Z}^d, \ell^{\infty})$ have this property. 
We also note that Theorem 2 is optimal: consider the indicator function $f = \chi_{\left\{v_1, \dots, v_N\right\}}$, then
$$  \| \nabla f\|_{L^1}  =  \# \partial_E (\left\{v_1, \dots, v_N\right\})$$
and this identity can now be applied twice, a second time to the set of vertices minimizing $ \#~ \partial_E (A)$ to deduce that the inequality cannot be improved in general.

  \subsection{$L^{\infty}-$P\'olya-Szeg\H{o}} We propose a condition that can be used to show that certain rearrangements satisfy  $\| \nabla f^*\|_{L^{\infty}} \leq c  \| \nabla f\|_{L^{\infty}}$ for some constant $c \in \mathbb{N}$. In contrast to the $L^1-$theory, which was concerned with edge-isoperimetric problems, the $L^{\infty}-$theory requires vertex-isoperimetry: given a graph $G$ and an integer $k \in \mathbb{N}$, we define the vertex-isoperimetric profile $\partial_{V}(k)$ as the largest integer such that any set $A \subset V$ with $\# A = k$ vertices is adjacent to at least $\partial_{V}(k)$ other vertices.
  
  \begin{theorem}[$L^{\infty}-$P\'olya-Szeg\H{o}]  Let $G= (V,E)$ be an infinite graph and assume that $\partial_{V}:\mathbb{N} \rightarrow \mathbb{N}$ is non-decreasing.
 Suppose $v_1, v_2, \dots$ is a permutation of the vertices so that, for some $c \in \mathbb{N}$ and all $N \in \mathbb{N}$ 
   $$ \left\{v \in V:  \min_{1 \leq i \leq N} d(v, v_i) \leq 1 \right\} \subseteq \left\{v_1, v_2, \dots, v_{ N + c \cdot \partial_{V}(N)} \right\}.$$
Then the rearrangement associated to the permutation  $v_1, v_2, \dots$ satisfies
  $$\| \nabla f^*\|_{L^{\infty}} \leq c \cdot  \| \nabla f\|_{L^{\infty}}.$$
  \end{theorem}
 The condition in Theorem 3 can be summarized as follows: when considering the set of vertices $\left\{v_1, \dots, v_N\right\}$, this
 set is guaranteed to be adjacent to at least $\partial_V(N)$ other vertices (and it might be adjacent to many more). In particular, there
 is at least one vertex $v_j$ adjacent to $\left\{v_1, \dots, v_N\right\}$ with $j \geq N + \partial_V(N)$. Theorem 3 guarantees that
 if there is an inverse inequality up to a multiplicative constant, meaning that if none of the adjacent indices exceed $N + c \cdot \partial_V(N)$ for some $c \in \mathbb{N}$, then this implies an $L^{\infty}-$P\'olya-Szeg\H{o} inequality with the same constant $c$.

\subsection{The full range.} We conclude with a setting for which we can obtain the full range $1 \leq p \leq \infty$ of the P\'olya-Szeg\H{o} inequality. Let $v_1, v_2, \dots$ be an arbitrary permutation of the vertices. The vertex-isoperimetric profile and the definition of edge- and vertex-neighborhood are naturally related and 
$$  \partial_V(N) = \inf_{A \subset V \atop \# A = N}  \# \partial_V \left( A\right) \leq  \# \partial_V \left( \left\{v_1, \dots, v_N\right\} \right) \leq  \# \partial_E \left( \left\{v_1, \dots, v_N\right\} \right).$$
If all three quantities coincide and the vertex-neighborhood is optimally arranged, then the rearrangement satisfies the P\'olya-Szeg\H{o} inequality for all $1 \leq p \leq \infty$.
\begin{theorem} Let $G= (V,E)$ be a graph and let $v_1, v_2, \dots$ be a permutation of the vertices so that for all $N \in \mathbb{N}$ we have both
$ \partial_V(N) =   \# \partial_E \left( \left\{v_1, \dots, v_N\right\} \right)$ and
$$  \partial_E \left( \left\{v_1, \dots, v_N\right\} \right) \subseteq \left\{v_1, \dots, v_{N+ \partial_V(N)} \right\}.$$
Then the associated rearrangement satisfies, for all $1 \leq p \leq \infty$ ,
$$ \| \nabla f^* \|_{L^p} \leq \| \nabla f \|_{L^p}.$$
\end{theorem}

    As a first example we revisit the canonical reordering on the lattice $\mathbb{Z}$ (see Fig. \ref{fig:line}) for which it is known that 
 $\| \nabla f^*\|_{L^p} \leq \| \nabla f\|_{L^p}$ for all $1 \leq p \leq \infty$ (see \cite{haj}). The permutation satisfies $ \partial_V(N) =  2 =  \# \partial_E \left( \left\{v_1, \dots, v_N\right\} \right)$ and the neighbors of the first $N$ elements are the first $N+2$ elements, thus Theorem 4 applies.
    
      \begin{center}
  \begin{figure}[h!]
  \begin{tikzpicture}[scale=1]
\draw [thick] (0,0) -- (7,0);
\filldraw (3.5, 0) circle (0.06cm);
\filldraw (4.5, 0) circle (0.06cm);
\filldraw (2.5, 0) circle (0.06cm);
\filldraw (5.5, 0) circle (0.06cm);
\filldraw (1.5, 0) circle (0.06cm);
\filldraw (6.5, 0) circle (0.06cm);
\filldraw (0.5, 0) circle (0.06cm);
\node at (0.5, 0.3) {6};
\node at (1.5, 0.3) {4};
\node at (2.5, 0.3) {2};
\node at (3.5, 0.3) {1};
\node at (4.5, 0.3) {3};
\node at (5.5, 0.3) {5};
\node at (6.5, 0.3) {7};
  \end{tikzpicture}
    \caption{A rearrangement on $\mathbb{Z}$ satisfying $\| \nabla f^*\|_{L^p} \leq \| \nabla f\|_{L^p}$.}
    \label{fig:line}
  \end{figure}
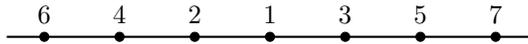
  \end{center}

  \subsection{Example: the regular tree}  
  Let $G$ be the infinite $d-$regular tree. There is a canonical order where the largest value of the function $f$ is sent to
  the root, the next $d$ values are distributed over the children of the root, the $d-1$ values after that are attached to the children of the vertex labeled 2 and so on (see Fig. \ref{fig:tree}).

    \begin{corollary}
    Let $G$ be an infinite $d-$regular tree and let $f^*$ denote the rearrangement corresponding to the canonical ordering. Then, for all $1 \leq p \leq \infty$
  $$ \| \nabla f^*\|_{L^p} \leq  \| \nabla f\|_{L^p}.$$
    \end{corollary}
    This extends a result of Pruss \cite{pruss} who showed that the canonical rearrangement satisfies 
  $ \| \nabla f^*\|_{L^2} \leq  \| \nabla f\|_{L^2}.$ 
The proof shows how the canonical ordering on the infinite regular tree is optimal with respect to both edge isoperimetry and vertex isoperimetry.

      \subsection{Example: competing rearrangements.} 
    As another example to illustrate the applicability of the results, we consider the graph $\mathbb{N} \times \left\{0,1\right\}$ with two vertices being connected if their Hamming distance differs by 1, see Fig. \ref{fig:spiral2}. There are at least two different rearrangements that appear to be somewhat natural on this graph: they are shown in Fig. \ref{fig:spiral2} and we will refer to them as the snake rearrangement and the lexicographic rearrangement.

  \begin{center}
  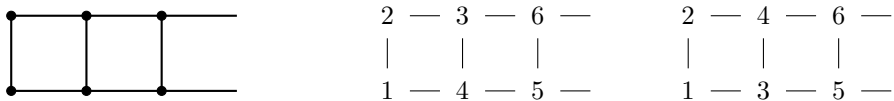
\begin{figure}[h!]
  \begin{tikzpicture}[scale=1]
  \filldraw (0,0) circle (0.06cm);
    \filldraw (1,0) circle (0.06cm);
  \filldraw (0,1) circle (0.06cm);
  \filldraw (1,1) circle (0.06cm);
  \filldraw (2,0) circle (0.06cm);
  \filldraw (2,1) circle (0.06cm);
\draw [thick] (0,0) -- (3,0);
\draw [thick] (0,1) -- (3,1);
\draw [thick] (0,0) -- (0,1);
\draw [thick] (1,0) -- (1,1);
\draw [thick] (2,0) -- (2,1);
  \node at (5,0) {1};
    \node at (5,1) {2};
    \node at (6,1) {3};
  \node at (6,0) {4};
    \node at (7,0) {5};
      \node at (7,1) {6};
  \draw [-] (5, 0.3) -- (5, 0.7);
  \draw [-] (5.3, 1) -- (5.7, 1);
    \draw [-] (5.3, 0) -- (5.7, 0);
  \draw [-] (6, 0.7) -- (6, 0.3);
 \draw [-] (6.3, 0) -- (6.7, 0);    
  \draw [-] (6.3, 1) -- (6.7, 1);   
  \draw [-] (7, 0.3) -- (7, 0.7);   
 \draw [-] (7.3, 1) -- (7.7, 1);  
  \draw [-] (7.3, 0) -- (7.7, 0);  
  \node at (9,0) {1};
    \node at (9,1) {2};
    \node at (10,1) {4};
  \node at (10,0) {3};
    \node at (11,0) {5};
      \node at (11,1) {6};              
  \draw [-] (9, 0.3) -- (9, 0.7);
  \draw [-] (9.3, 1) -- (9.7, 1);
    \draw [-] (9.3, 0) -- (9.7, 0);
  \draw [-] (10, 0.7) -- (10, 0.3);
 \draw [-] (10.3, 0) -- (10.7, 0);    
  \draw [-] (10.3, 1) -- (10.7, 1);   
  \draw [-] (11, 0.3) -- (11, 0.7);   
 \draw [-] (11.3, 1) -- (11.7, 1);  
  \draw [-] (11.3, 0) -- (11.7, 0);               
    \end{tikzpicture}
    \caption{The graph $(\mathbb{N} \times \left\{0,1\right\}, \ell^1)$ (left), the snake rearrangement (middle) and the lexicographic rearrangement (right).}
    \label{fig:spiral2}
  \end{figure}
  \end{center}
    \vspace{-10pt}
  
  Applying Theorem 2 and Theorem 3 allows us to quickly derive suitable bounds for both rearrangements. The snake rearrangement
  is well behaved with respect to the edge-isoperimetric properties and satisfies the Poly\'a-Szeg\H{o} inequality in $L^1$. The 
  lexicographic rearrangement is well-behaved in both $L^1$ and $L^{\infty}$.
  
  \begin{corollary} The snake rearrangement satisfies
  $$ \| \nabla f^*\|_{L^1} \leq  \| \nabla f\|_{L^1} \qquad \mbox{and} \qquad  \| \nabla f^*\|_{L^{\infty}} \leq  2 \cdot \| \nabla f\|_{L^{\infty}}.$$
The lexicographic rearrangement satisfies
  $$ \| \nabla f^*\|_{L^1} \leq  \| \nabla f\|_{L^1} \qquad \mbox{and} \qquad  \| \nabla f^*\|_{L^{\infty}} \leq  \| \nabla f\|_{L^{\infty}}.$$
  \end{corollary}
  
  These examples illustrate how different rearrangements on the same graph can lead to different outcomes and how the size of the constant $c_p$ in $\| \nabla f^*\|_{L^p} \leq c_p\| \nabla f\|_{L^p}$ can be used as an implicit measure of quality of the rearrangement.

    \section{Proofs}
We start by establishing Theorem 2  in \S \ref{sec:2} and Theorem 3 in \S \ref{sec:3}.  These will then be used to prove Theorem 1 in \S \ref{sec:1}. \S \ref{sec:short} contains a short proof of the result of Hajaiej, Han \& Hua \cite{hhh} showing that no rearrangement on $(\mathbb{Z}^2, \ell^1)$ can satisfy the inequality for $p=2$ (this argument is independent of the others). \S \ref{sec:4} contains a proof of Theorem 4 which is then used to prove Corollary 1 in \ref{cor1}. Finally, \S \ref{cor2} gives a proof for Corollary 2. 

    \subsection{Proof of Theorem 2} \label{sec:2}
    \begin{proof}
    We assume without loss of generality that $\|f\|_{L^{\infty}} = 1$.
  By definition
  $$\| \nabla f\|_{L^1}= \sum_{(v,w) \in E} |f(v) - f(w)|.$$
  We define, for each $0 \leq s \leq 1$, the superlevel set 
  $$ \left\{f \geq s\right\} =  \left\{v \in V: f(v) \geq s \right\} \subseteq V.$$
 We will use the coarea formula, already using that without loss of generality $\|f\|_{L^{\infty}}=1$, in the form (see, for example, \cite[Lemma 1]{tillich})
    $$\| \nabla f\|_{L^1} = \int_0^{1}  \int_{\partial  \left\{f \geq s\right\} }  1 ~dx ds = \int_0^1 \# \partial  \left\{f \geq s\right\}  ds.$$
    This coarea formula can be seen as follows: suppose $(v,w) \in E$. This edge contributes $|f(v) - f(w)|^{}$
    to the left-hand side while contributing  $1$ over an interval of length $|f(v) - f(w)|$
    on the right-hand side. We note that, by definition,
 \begin{align*}
    \int_0^{1} \# \partial \left\{f \geq s\right\}    ds \geq   \int_0^{1}  \inf_{A \subset V \atop \# A = \#\left\{f \geq s\right\} } \# \partial_E (A)  ds.
  \end{align*}
By definition of the rearrangement, we have $ \#\left\{f \geq s\right\} =  \#\left\{f^* \geq s\right\}.$
As an assumption of Theorem 2,
$$ \# \partial_E (\left\{v_1, \dots, v_N\right\}) \leq \beta +   \alpha  \inf_{ A \subseteq  V \atop  \# A =N}~~ \#~ \partial_E (A),$$
from which we deduce, applying the coarea formula once more,
 \begin{align*}
 \|\nabla f\|_{L^1} \geq \int_0^{1}  \inf_{A \subset V \atop \# A = \#\left\{f \geq s\right\} } \# \partial_E (A)  ds &=   \int_0^{1}  \inf_{A \subset V \atop \# A = \#\left\{f^* \geq s\right\} } \# \partial_E (A)  ds \\
  &\geq  \frac{1}{\alpha} \int_0^{1}  \# \partial \left\{f^* \geq s\right\}  -\beta~ ds\\
  &=  - \frac{\beta}{\alpha} +  \frac{1}{\alpha}\| \nabla f^* \|_{L^1}.
    \end{align*}
This can rearranged as
$$\| \nabla f^* \|_{L^1} \leq \beta + \alpha \| \nabla f\|_{L^1}.$$
      which, recalling the normalization $\|f\|_{L^{\infty}} = 1$ now implies the result.
\end{proof}

\subsection{Proof of Theorem 3}\label{sec:3}
\begin{proof}
We assume without loss of generality that $\| \nabla f\|_{L^{\infty}} = 1$. Suppose that $v_1, v_2, \dots$ is a rearrangement with the desired properties. We want to show that for all $i \geq 1$, if the vertex $v_i$ has a neighbor $(v_i, v_j) \in E$, then the value of $f^*$ in $v_j$ is not much smaller than $f^*(v_i)$. Formally:
$$ \inf \left\{ f^*(v_j): j \geq i ~\mbox{and}~ (v_i, v_j) \in E \right\} \geq f^*(v_i) - c.$$
We note that the function $f$ has few large values in the sense that it assumes a value at least as large as $f^*(v_i)$ a fixed number of times: by definition, $f^*(v_i)$ is the $i-$th largest value and thus
$$ \# \left\{v \in V: f(v) \geq f^*(v_i) \right\} \geq i.$$
We note that the inequality need not be strict, it could be that the value $f^*(v_i)$ is attained many more times. 
We call this set $A =   \left\{v \in V: f(v) \geq f^*(v_i) \right\}$. By definition of the vertex-isoperimetric profile and the fact that it is non-decreasing (which is one of the assumptions), we have
$$ \# \left\{ v \in V: d(v, A) \leq 1 \right\} \geq i + \partial_V(\#A) \geq  i + \partial_V(i).$$
If $\partial_V$ is monotonically non-decreasing then, for any $k \geq 1$,
\begin{align*}
 \# \left\{ v \in V: d(v, A) \leq k+1 \right\} &\geq  \# \left\{ v \in V: d(v, A) \leq k \right\} \\
 &+ \partial_V\left(  \# \left\{ v \in V: d(v, A) \leq k \right\} \right) \\
 &\geq  \# \left\{ v \in V: d(v, A) \leq k \right\} + \partial_V(\#A) \\
 &\geq  \# \left\{ v \in V: d(v, A) \leq k \right\} + \partial_V(i)
\end{align*}
from which one obtains, by iteration,
$$ \# \left\{ v \in V: d(v, A) \leq k \right\} \geq i + k \cdot \partial_V(i).$$
We deduce that, for any arbitrary integer $k \geq 1$, there are at least $ i + k \cdot \partial_V(i)$
vertices $w$ satisfying $f(w) \geq f^*(v_i) - k$ (since $\| \nabla f\|_{L^{\infty}} = 1$, we know $f$ can only decrease by at most 1 each step)
 and thus the $(i + k \cdot \partial_V(i))-$th largest value assumed by $f$ is at least $f^*(v_i) - k$.
By assumption, there is $c>0$ such that $v_i$ is not connected to vertices with a much larger index: $i \leq j$ and $(v_i, v_j) \in E$, then we have
$ j \leq i + c \cdot \partial_V(i)$. 
Setting $k=c$ implies that there are at least $ i + c \cdot \partial_V(i)$ vertices on which the function is at least $f^*(v_i) - c$ and thus, since $f^*(v_j)$ is the $j-$th largest value assumed by the function and $ j \leq i + c \cdot \partial_V(i)$, we have
 $f^*(v_j) \geq f^*(v_i) - c$. We deduce
 $$ |f^*(v_i) - f^*(v_j)| \leq c = c \cdot \| \nabla f\|_{L^{\infty}}$$
 and since $i \leq j$ were arbitrary (subject to $(i,j) \in E$), we arrive at the result.
\end{proof}

    \subsection{Proof of Theorem 1}  \label{sec:1}
   
        The proof of Theorem 1 consists in showing that the spiral rearrangement leads to a reordering of the vertices of $(\mathbb{Z}^2, \ell^1)$ such that Theorem 2 is applicable with $\alpha=1$, $\beta = 0$ and Theorem 3 is applicable with $c=2$.\\
                
        \textbf{The $L^1-$inequality.} The first step amounts to showing that the first $n$ elements of the spiral rearrangement have as few neighbors as possible for a set of that size. These problems have been solved in much greater generality and in more difficult settings, see Bollob\'as \& Leader \cite{boll, boll2}. We include a self-contained elementary argument for this much simpler special case.        
        
  \begin{proof}
Let $A \subset \mathbb{Z}^2$ be an arbitrary set on $n$ elements. We can project the set onto the $x-$axis, i.e.
$ A_x = \left\{x \in \mathbb{Z}: \exists y \in \mathbb{Z} ~\mbox{such that}~ (x,y) \in A \right\}$ and likewise onto the
$y-$axis leading to $A_y$. We have $A \subseteq A_x \times A_y$ and thus $\# A_x \cdot \# A_y \geq n$. Each
element in $A_x$ identifies at least two unique edges between $A$ and $A^c$ (at the top and the bottom of that slice)
  and, likewise, each element in $A_y$ identifies at least two unique edges (from the left and right end of the slice). Thus
  \begin{align*}
   \# \partial_E(A) &\geq 2 \#A_x + 2 \# A_y \geq  4 \sqrt{\# A_x} \sqrt{\# A_y}  \geq 4 \sqrt{n}.
   \end{align*}
Equality in the second inequality can only occur if $\#A_x = \# A_y$ and equality in the last inequality can only happen if $\# A_x \# A_y = n$. Moreover, since all three numbers $\#A_x, \#A_y, n$ are integers, one can get a little extra information out of the inequality which turns out to be sufficient.  We first observe that the inequality immediately implies optimality of the spiral rearrangement whenever
   $n$ is a square number. Let us now fix $n$ to be a square number and consider the case of $n + k$ where $1 \leq k < 2\sqrt{n} + 1$. Assuming w.l.o.g. that $\# A_x = \sqrt{n} - \ell$ for some $\ell \geq 0$ implies $\#A_y \geq \sqrt{n} + \ell + 1$ since, if $\#A_y \leq \sqrt{n} + \ell$, then $\# A_x \# A_y < n$ which is a contradiction. Thus $\#A_y \geq \sqrt{n} + \ell + 1$ which implies that $ \# \partial_E(A) \geq 4 \sqrt{n} + 2$. The spiral construction has
   exactly $4\sqrt{n} + 2$ neighbors as long as $k \leq \sqrt{n}$ and is thus optimal in this range. It remains to analyze the case where $\sqrt{n} < k \leq 2\sqrt{n}$. In that case, assuming w.l.o.g. $\# A_x = \sqrt{n} - \ell$ for some $\ell \geq 0$, we see that  $\#A_y \geq \sqrt{n} + \ell + 2$ since, as above,
   if we had $\#A_y \leq \sqrt{n} + \ell + 1$, then $\# A_x \# A_y \leq n - \ell^2 + \sqrt{n} - \ell \leq n + \sqrt{n}$ which is another contradiction. $\#A_y \geq \sqrt{n} + \ell + 2$ implies $ \# \partial_E(A) \geq 4 \sqrt{n} + 4$ matching again the spiral construction.
    \end{proof}

\hspace{1pt}\\

\textbf{The $L^{\infty}-$inequality.}
\begin{proof}
Our goal is to show that Theorem 3 is applicable with constant $c=2$. This means that we want to establish that, for the spiral arrangement, 
$$ \partial_V \left( \left\{v_1, \dots, v_N \right\}\right) \subseteq \left\{v_1, \dots, v_{N + 2 \cdot \partial_V(N)} \right\}.$$
This requires us to analyze the size of $\partial_V(N)$ and to understand the neighborhood of the first $N$ elements in
the spiral embedding $\partial_V \left( \left\{v_1, \dots, v_N \right\}\right)$. The neighborhood question is easiest, we see
that the spiral embedding satisfies 
$$ \partial_V \left( \left\{v_1, \dots, v_N \right\}\right) \subseteq \left\{v_1, \dots, v_{M} \right\},$$
where, asymptotically to leading order, $M \sim N + 4\sqrt{N}$. 
The vertex-isoperimetric problem on $(\mathbb{Z}^2, \ell^1)$ has
been solved by Wang \& Wang \cite{wang} (they solve the problem in all dimensions), the asymptotically optimal shape is asymptotically an $\ell^1-$ball (in contrast to the edge-isoperimetric problem
where the optimal shape is an $\ell^{\infty}-$ball). 

\begin{figure}[h!]
\begin{minipage}[l]{.45\textwidth}
\begin{tikzpicture}
\node at (-3,0) {};
\node at (0,0) {1};
\node at (0,0.5) {2};
\node at (0.5,0) {3};
\node at (-0.5,0) {4};
\node at (0,-0.5) {5};
\node at (0.5,0.5) {6};
\node at (-0.5,0.5) {7};
\node at (0,1) {8};
\node at (1,0) {9};
\node at (0.5,-0.5) {10};
\node at (-1,0) {11};
\node at (-0.5,-0.5) {12};
\node at (0,-1) {13};
\end{tikzpicture}
\end{minipage}
\begin{minipage}[r]{.5\textwidth}
\begin{center}
\begin{tabular}{c  c  c  c c  c c}
$N$ & 1 & 2 & 3 & 4 & 5 & 6\\
$\# \partial_V(N)$ &  4  &  6 &  7 & 8 & 8 & 9\\
\end{tabular}
\end{center}
\begin{center}
\begin{tabular}{c  c  c  c c  c c}
$N$ & 1 & 2 & 3 & 4 & 5 & 6\\
$M$ &  8  &  11 &  14 & 15 & 18 & 19
\end{tabular}
\end{center}
\end{minipage} 
\caption{Left: a sequence of nested minimizers of the vertex-isoperimetry problem (Wang \& Wang \cite{wang}). Right: values of $\# \partial_V(n)$ for small $n$ as well as the smallest $m$ such that the neighbors of the first $n$ elements are contained
in the first $m$ elements.} 
\label{fig:smaln}
\end{figure}

The asymptotic $\ell^1-$ball with $N$ elements has $\partial_V(N) \sim 4\sqrt{N}$ neighbors.
This shows that things asymptotically match (even with asymptotic constant $c=1$).  It remains to analyze the case of small values. This is done in Fig. \ref{fig:smaln}: we see that $c=1 + \varepsilon$ is sufficient asymptotically and that $c=2$ is enough to ensure that the initial values satisfy $M \leq N + c \cdot \partial_V(N)$ and thus
$$ \| \nabla f^*\|_{L^{\infty}} \leq 2 \| \nabla f\|_{L^{\infty}}.$$
\end{proof}

  \begin{figure}[h!]
\begin{minipage}[l]{.45\textwidth}
\begin{tikzpicture}
\node at (-3,0) {};
\node at (0,0) {2};
\node at (0,0.5) {1};
\node at (0.5,0) {1};
\node at (-0.5,0) {1};
\node at (0,-0.5) {1};
\node at (0.5,0.5) {0};
\node at (-0.5,0.5) {0};
\node at (0.5,-0.5) {0};
\node at (-0.5,-0.5) {0};
\node at (-1, -1) {$f$};
\end{tikzpicture}
\end{minipage}
\begin{minipage}[r]{.5\textwidth}
\begin{tikzpicture}
\node at (-3,0) {};
\node at (0,0) {2};
\node at (0,0.5) {1};
\node at (0.5,0) {1};
\node at (-0.5,0) {0};
\node at (0,-0.5) {0};
\node at (0.5,0.5) {1};
\node at (-0.5,0.5) {1};
\node at (0.5,-0.5) {0};
\node at (-0.5,-0.5) {0};
\node at (-1, -1) {$f^*$};
\end{tikzpicture}
\end{minipage} 
\caption{An example showing $ \| \nabla f^*\|_{L^{\infty}} \leq 2 \| \nabla f\|_{L^{\infty}}$ is optimal for the spiral rearrangement.}
\end{figure}

  \subsection{Rearrangements on $(\mathbb{Z}^2, \ell^1)$ in $L^2$} \label{sec:short} We give a short proof that no rearrangement on the grid graph $(\mathbb{Z}^2, \ell^1)$ can satisfy $\| \nabla f^*\|_{L^2} \leq \| \nabla f\|_{L^2}$. This result was recently proven by Hajaiej, Han \& Hua \cite{hhh}, our proof is heavily inspired by theirs.
  
\begin{proposition}
Let $1 < p < \infty$. The lattice graph $(\mathbb{Z}^2, \ell^1)$ does not admit a rearrangement procedure such that $\| \nabla f^*\|_{L^p} \leq \| \nabla f\|_{L^p}$ for all functions $f$.
\end{proposition}
  
  \begin{proof} We first assume that there exists such a rearrangement procedure. By looking at a very particular function, we deduce where the first 5 terms have to be placed. We then show that a different function has increasing gradients under this procedure.
  Consider first a function assuming the non-zero values $(n,1,1,1,1)$ with $n \gg 1$ (and assuming value $0$ everywhere else). One natural way the values could be arranged is to have $n$ in the center be surrounded by 4 times the value $1$. This arrangement leads to a function $f_1$ with
$$ \| \nabla f_1\|_{L^p}^p =  4(n-1)^p + 12.$$
  Suppose we have any other arrangement $f_2$: then the largest value $n$ is surrounded by at most 3 times the value 1 and thus
  $$ \| \nabla f_2\|_{L^p}^p \geq n^p + 3(n-1)^p =  4(n-1)^p + (n^p - (n-1)^p).$$
We see that for $p>1$ and $n$ sufficiently large, we have $ \| \nabla f_1\|_{L^p}^p < \| \nabla f_2\|_{L^p}^p$ which shows that the optimal arrangement, if it exists, has to be the one shown in Fig. \ref{fig:shortproof}. On the other hand, if we consider a function assuming the values $(1,1,1,1,1)$ (and $0$ everywhere else), then it is easily seen that the same central rearrangement leads to an energy of $\|\nabla f_1\|_{L^p}^p = 12$ while a suitable asymmetric `block-'rearrangement leads to $\|\nabla f_2\|_{L^p}^p = 10$.
       \end{proof}
  \begin{center}
  \begin{figure}[h!]
  \begin{tikzpicture}[scale=0.5]
  \node at (0,0) {$n$};
    \node at (1,0) {$1$};
    \node at (1,1) {0};
  \node at (0,1) {1};
  \node at (-1,1) {0};
  \node at (-1,0) {1};
  \node at (-1,-1) {0};
  \node at (0,-1) {1};
  \node at (1,-1) {0};
  \end{tikzpicture}
    \begin{tikzpicture}[scale=0.9]
  \node at (0,0) {};
    \node at (1,0) {};
      \end{tikzpicture}
    \begin{tikzpicture}[scale=0.5]
  \node at (0,0) {$1$};
    \node at (1,0) {$1$};
    \node at (1,1) {1};
  \node at (0,1) {1};
  \node at (-1,1) {0};
  \node at (-1,0) {0};
  \node at (-1,-1) {0};
  \node at (0,-1) {1};
  \node at (1,-1) {0};
  \end{tikzpicture}
  \caption{Two competing configurations (invisible lattice points have value 0): the configuration on the left is an optimal rearrangement on the grid graph $\mathbb{Z}^2$ for $n$ sufficiently large. However, the configuration on the right does better when $n=1$ which shows that no universal rearrangement satisfying $\| \nabla f^*\|_{L^2} \leq \| \nabla f\|_{L^2}$ exists.}
    \label{fig:shortproof}
  \end{figure}
  \end{center}

We quickly note that the proof can be quantified to yield the following.
\begin{proposition} Given an arbitrary rearrangement on $(\mathbb{Z}^2, \ell^1)$, there always exists a function $f:\mathbb{Z}^2 \rightarrow \mathbb{R}_{\geq 0}$ such that
$$ \| \nabla f^*\|_{L^2} \geq 1.01 \cdot \| \nabla f\|_{L^2} $$
\end{proposition}
\begin{proof} We consider again a function assuming values $(n,1,1,1,1)$ and the value 0 everywhere else. If the rearrangement happens to be the one that places the largest value in the center and the four next values around it, then we choose $n=1$ and compare to the right configuration in
 Fig. \ref{fig:shortproof} and obtain
 $$ \| \nabla f^*\|_{L^2}^2 = 12 = \frac{6}{5} \cdot 10 = \frac{6}{5} \cdot \| \nabla f\|_{L^2}^2.$$
 If the rearrangement is of any other type, then we assume $f$ to be as in the left configuration of  Fig. \ref{fig:shortproof}. The largest value in the rearrangement is surrounded by at least one square containing the, at most sixth largest value, which happens to be 0 in our construction. We obtain 
 $$    \| \nabla f^*\|_{L^2}^2 \geq n^2 + 3(n-1)^2 + 5.$$
Comparing to the original function with $ \| \nabla f\|_{L^2}^2 = 4 n^2 - 8n + 16$ and optimizing in $n$ leads to the result.
\end{proof}
It would be nice if the constant in the estimate could be improved: it can be said to measure the impossibility of a P\'olya-Szeg\H{o} inequality, see also \cite{shubham}.

\subsection{Proof of Theorem 4} \label{sec:4}
The standard coarea formula says that
$$ \| \nabla f\|_{L^p}^p =  \int_{0}^1 \int_{ \partial_{E} \left\{ f \geq s \right\} } \left| \nabla f\right|^{p-1} dx ds$$
and was already used in the proof of Theorem 2 for $p=1$: the idea being each edge contributes $|f(v) - f(w)|^p$ to the
left-hand side and $|f(v) - f(w)^{p-1}$ over an interval of length $|f(v) - f(w)|$. We will now use a small modification
of the idea: the advantage of this new formulation is that the values in $\left\{ f \geq s \right\}$ no longer show
up in the inner integral which is solely determined by $s$ and the values outside.

\begin{lemma}[Modified Coarea Formula] Suppose $\|f\|_{L^{\infty}} = 1$ and $1 \leq p < \infty$. Then
$$ \| \nabla f\|_{L^p}^p =  p\int_{0}^1 \int_{ \partial_{E} \left\{ f \geq s \right\} } \left| \nabla \min(f, s) \right|^{p-1} dx ds.$$
\end{lemma}
\begin{proof} We consider again a single edge $(v,w) \in E$. The contribution to the left-hand side is
$|f(v) - f(w)|^p$. 
The contribution to the right-hand side is
$$ \int_{\min\left\{f(v), f(w) \right\}}^{\max\left\{f(v), f(w) \right\}} \left(s - \min\left\{f(v), f(w) \right\}\right)^{p-1} ds = \frac{|f(v) - f(w)|^p}{p}.$$
\end{proof}

\begin{proof}[Proof of Theorem 4] Let now $f:V \rightarrow \mathbb{R}_{\geq 0}$ be normalized to $\|f\|_{L^{\infty}} = 1$ but otherwise arbitrary. The modified coarea formula is
$$ \| \nabla f\|_{L^p}^p =  p\int_{0}^1 \int_{ \partial_{E} \left\{ f \geq s \right\} } \left| \nabla \min(f, s) \right|^{p-1} dx ds.$$
Let us now fix an arbitrary value $0 < s  <  1$ and analyze the inner integral
$$J = \int_{ \partial_{E} \left\{ f \geq s \right\} } \left| \nabla \min(f, s) \right|^{p-1} dx.$$
We may think of the function values of $f$ as sorted
$$ f^*(v_1) \geq f^*(v_2) \geq \dots \geq f^*(v_j) \geq s > f^*(v_{j+1}) \geq f^*(v_{j+2}) \geq \dots$$
The integral $J$ runs at least over
$ \# \partial_{E} \left\{ f \geq s \right\} \geq \# \partial_{V} \left\{ f \geq s \right\}$
different edges.  Abbreviating $k = \# \partial_{V} \left\{ f \geq s \right\}$, we deduce
$$  \int_{ \partial_{E} \left\{ f \geq s \right\} } \left| \nabla \min(f, s) \right|^{p-1} dx \geq \sum_{i=1}^{k} (s - f^*(v_{j+i}))^{p-1}.$$
At the same time, since by assumption
$$ \#\partial_{V} \left\{ f \geq s \right\} \geq \# \partial_{V} \left\{ f^* \geq s \right\} =  \# \partial_{E} \left\{ f^* \geq s \right\}$$
as well as $  \partial_E \left( \left\{v_1, \dots, v_N\right\} \right) \subseteq \left\{v_1, \dots, v_{N+ \partial_V(N)} \right\}$,
we deduce
$$  \sum_{i=1}^{k} (s - f^*(v_{j+i}))^{p-1} =  \int_{ \partial_{E} \left\{ f^* \geq s \right\} } \left| \nabla \min(f^*, s) \right|^{p-1} dx $$
from which it follows that
\begin{align*}
  \| \nabla f\|_{L^p}^p  &= p\int_{0}^1 \int_{ \partial_{E} \left\{ f \geq s \right\} } \left| \nabla \min(f, s) \right|^{p-1} dx ~ds \\
  &\geq p\int_{0}^1 \int_{ \partial_{E} \left\{ f^* \geq s \right\} } \left| \nabla \min(f^*, s) \right|^{p-1} dx ~ds = \| \nabla f^*\|_{L^p}^p.
 \end{align*}
\end{proof}

  \subsection{Proof of Corollary 1} \label{cor1}
  We start with a Lemma which probably exists somewhere in the literature. However, we were unable to locate the statement for the vertex expansion and therefore add a quick argument.
  \begin{lemma} Let $G$ be the infinite $d-$regular tree. Then, for any subset $A \subset V$ of vertices, the edge boundary satisfies
      $$\# \partial_E(A) \geq (d-2) \# A  +2$$
and $A$ is adjacent to at least
        $$\# \partial_V(A) \geq (d-2) \# A  +2$$
vertices and these bounds are best possible.
  \end{lemma}
  \begin{proof}  It is clear that if $A$ is connected, then the tree-structure implies that each outgoing edge goes to a unique vertex and thus, in that case, $\# \partial_E(A) = \# \partial_V(A)$. It remains to show that sets minimizing the vertex-neighborhood are connected. We prove, using induction on $n$, that any set of $n$ vertices minimizing the number of adjacent vertices has to connected. The statement is vacuous for $n=1$. For $n=2$, note that two connected vertices have $2(d-1)$ neighbors while two vertices that are not connected have at least $2d -1$ neighbors (with equality if and only if they are distance 2). Let now $n$ be arbitrary, pick an arbitrary vertex $v_0 \in A$ to be the root of the tree (for navigational purposes) and pick $v_1 \in A$ to be a vertex that has the largest distance from $v_0$. The set $A\setminus \left\{v_1\right\}$ has, by induction assumption, at least $(d-2)(\#A -1) + 2$ neighbors with equality only if $A$ is connected: moreover, none of the $d-1$ neighbors of $v_1$ that are distance $d(v_0, v_1) + 1$ from $v_0$ can be neighbors of $A\setminus v_1$. Adding the point $v_1$ back, we recover at least those $d-1$ neighbors while removing one neighbor from the set of neighbors ($v_1$ itself) if and only if $A$ is connected.
  \end{proof}
        \begin{proof}[Proof of Corollary 1]
        The Lemma implies that
        $$ \partial_V(N) = (d-2)N + 2 $$
        while the canonical reordering can be seen to satisfy
        $$   \# \partial_E \left( \left\{v_1, \dots, v_N\right\} \right) = (d-2)N+2$$
        as follows: clearly, for the root, we have $   \# \partial_E \left( \left\{v_1, \dots, v_N\right\} \right) = d$. Then, whenever adding a new point, we add $(d-1)$ new boundary vertices while removing one of the existing boundary vertices (exactly the new point that has been added).
        Hence the first condition of Theorem 4. is satisfied. It suffices to prove 
        $$  \partial_E \left( \left\{v_1, \dots, v_N\right\} \right) \subseteq \left\{v_1, \dots, v_{N+ (d-2)N + 2} \right\}.$$

  Fixing an arbitrary vertex $v_1$ and denoting it to be the root, we see that there are $d \cdot (d-1)^{k-1}$ vertices at distance $k$. We distinguish two cases: the first case is that $2 \leq N \leq d+1$. An explicit computation shows that
    $$ \partial_V \left\{v_1, \dots, v_N\right\} = \left\{v_1, \dots, v_{d+1 + (N-1)(d-1)} \right\}$$
    and, as required by Theorem 4
    $ d+1 + (N-1)(d-1) \leq N + (d-2)N + 2.$
    Let us now assume that $N \geq d+2$. Let us introduce $k= d(v_0, v_N)$ as the distance to the root $v_0$. Since $N \geq d+2$, we have $k \geq 2$ and thus we can narrow down the possible value of $N$ in terms of $N_{k-1}$, the number of points at distance at most $k-1$, and $N_k$, the number of points at distance at most $k$ via
    $$ N_{k-1} = 1 + d +d \sum_{j=0}^{k-2} (d-1)^j < N \leq 1 + d +d \sum_{j=0}^{k-1} (d-1)^j = N_k.$$
    The neighbors of $\left\{v_1, \dots, v_N\right\}$ are then $\left\{v_1, \dots, v_M\right\}$ where
 $M= N_k + (N-N_{k-1})(d-1).$
We would like to have $M \leq N + (d-2)N + 2 = (d-1)N + 2$. This inequality is equivalent to $N_k - (d-1) N_{k-1} \leq 2$ which is easily seen to be true (in fact, equality holds). Theorem 4. applies and Corollary 1 follows. \end{proof}

  \subsection{Proof of Corollary 2} \label{cor2}
        \begin{proof}
We first consider the edge-isoperimetric problem. If $\#A = 1$ and we are dealing with a single vertex, then $\# \partial_E A \geq 2$ with equality if and only if it is on the left boundary (vertex 1 or vertex 2 in Fig. \ref{fig:spiral2}). It now suffices to consider $\# A \geq 2$.
If $A$ is fully contained in the upper or the lower row, one can easily see that $\# \partial_E(A) \geq \# A + 1$ is quite large. It suffices to deal with the case where $A$ has both elements in the upper and lower row.
 It is easy to see (by staying in each row and going to infinity) that for any set $A$ with an even number of elements, the best possible bound is $\# \partial_E(A) \geq 2$. If $\#A$ is odd, then we can consider the number of elements in the lower row and the number of elements in the upper row and notice that one of them has to be even and the other one has to be odd leading to at least one edge between rows and two edges when going to $\infty$, thus $\# \partial_E(A) \geq 3$. Altogether, we have
 
           $$ \# \partial_E(A) \geq \begin{cases} 2 \qquad &\mbox{if}~\# A = 1 \\
          3 \qquad &\mbox{if}~\# A \geq 2 ~\mbox{and odd} \\
           2 \qquad &\mbox{if}~\# A ~\mbox{even} \\
           \end{cases}$$
 Both the snake and the lexicographic rearrangement satisfies exactly the same bounds and thus Theorem 2 is applicable and we deduce, for both rearrangements,
 $$ \| \nabla f^*\|_{L^{1}} \leq  \| \nabla f\|_{L^{1}}.$$
 As for vertex expansion, it is easy to see that $\# \partial_V(A) \geq 2$ and that for the proposed rearrangement, this is sharp at each step. Moreover, for the snake rearrangement we see that
 $$ \partial_V(\left\{1,2,\dots, N\right\}) \subseteq \left\{1,2,\dots, N+3 \right\}$$
 and since $3 \leq 2 \cdot 2$ we have, with Theorem 3,
  $ \| \nabla f^*\|_{L^{\infty}} \leq  2 \cdot \| \nabla f\|_{L^{\infty}}.$
 It is easy to see that this is optimal, see Fig. \ref{fig:exam1}.   
    
    \begin{center}
  \begin{figure}[h!]    \label{fig:exam1}
  \begin{tikzpicture}[scale=1]
  \node at (5,0) {2};
    \node at (5,1) {1};
    \node at (6,1) {0};
  \node at (6,0) {1};
    \node at (7,0) {0};
      \node at (7,1) {0};
  \draw [] (5, 0.3) -- (5, 0.7);
    \draw [] (5.3, 1) -- (5.7, 1);
        \draw [] (6, 0.7) -- (6, 0.3);
           \draw [] (6.3, 0) -- (6.7, 0);     
            \draw [] (7, 0.3) -- (7, 0.7);   
                 \draw [] (7.3, 1) -- (7.7, 1);  
            \draw [] (7.3, 0) -- (7.7, 0);          
   \draw [] (5.3, 0) -- (5.7, 0);
     \draw [] (6.3, 1) -- (6.7, 1);                       
  \node at (10,0) {2};
    \node at (10,1) {1};
    \node at (11,1) {1};
  \node at (11,0) {0};
    \node at (12,0) {0};
      \node at (12,1) {0};
  \draw [] (10, 0.3) -- (10, 0.7);
    \draw [] (10.3, 1) -- (10.7, 1);
        \draw [] (11, 0.7) -- (11, 0.3);
           \draw [] (11.3, 0) -- (11.7, 0);     
            \draw [] (12, 0.3) -- (12, 0.7);   
                 \draw [] (12.3, 1) -- (12.7, 1);  
            \draw [] (12.3, 0) -- (12.7, 0);          
   \draw [] (10.3, 0) -- (10.7, 0);
     \draw [] (11.3, 1) -- (11.7, 1);       
    \end{tikzpicture}
    \caption{ $f$ (left) and its snake rearrangement (right). We have $\| \nabla f^*\|_{L^1}  = 5 = \| \nabla f\|_{L^1} $ and $\| \nabla f^*\|_{L^{\infty}} = 2 = 2 \cdot \| \nabla f\|_{L^{\infty}}$.}
  \end{figure}
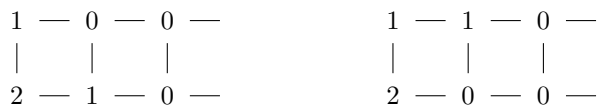
  \end{center}
    \vspace{-10pt}
  
The lexicographic rearrangement additionally satisfies
 $$ \partial_V(\left\{1,2,\dots, N\right\}) \subseteq \left\{1,2,\dots, N+2 \right\}$$
 which implies, with Theorem 3,
  $ \| \nabla f^*\|_{L^{\infty}} \leq   \| \nabla f\|_{L^{\infty}}$.
    \end{proof}
    
    We conclude by remarking that since both the snake rearrangement and the lexicographic rearrangement
  satisfy $ \| \nabla f^*\|_{L^{1}} \leq  \| \nabla f\|_{L^{1}}$ one can apply them both consecutively without changing the $L^1-$norm
  of the derivative even though the function is actually rearranged differently each time. This hints at
  an underlying symmetry and might be useful in practice (see the `technique of competing symmetries' in \cite{lieb}).
It stands to reason that examples of graphs that admit multiple different rearrangements satisfying the P\'olya-Szeg\H{o} inequality are probably rare.\\
  
\textbf{Acknowledgment.} I am grateful to a very diligent anonymous referee whose suggestions greatly improved the manuscript.

\end{document}